\documentclass[12pt,reqno]{amsart}

\usepackage{ccfonts}
\usepackage[T1]{fontenc}
\usepackage[cp1251]{inputenc}

\usepackage{amsmath,amsxtra,amssymb,eufrak}
\usepackage[all]{xy}
\usepackage{mathrsfs}
\usepackage{cite}
\usepackage{paralist}
\usepackage[hypertex]{hyperref}
\usepackage[active]{srcltx} 

\newtheorem{thm}{Theorem}
\newtheorem{lm}[thm]{Lemma}

\newtheorem{pr}[thm]{Proposition}

\theoremstyle{definition}

\newtheorem{rem}[thm]{Remark}

\newcommand*{\wh}{\widehat}

{\end{compactenum}}
\newcommand{\CC}{\mathbb{C}}

\newcommand{\R}{\mathbb{R}}
\newcommand{\Z}{\mathbb{Z}}
\newcommand{\N}{\mathbb{N}}

\renewcommand{\le}{\leqslant}
\renewcommand{\ge}{\geqslant}

\let \al         =\alpha

\let \de         =\delta

\let \te         =\theta        

\let \phi         =\varphi

\title{An analytic criterion for the local finiteness\\ of a countable semigroup}
\author{O.\,Yu.~Aristov}
\email{aristovoyu@inbox.ru}

\begin{document}

\begin{abstract}
We prove that a countable semigroup $S$ is locally finite if and only if the Arens-Michael envelope of its semigroup algebra is a $(DF)$-space. This is a counterpart to a recent result of the author, which asserts that $S$  is finitely generated if and only if the Arens-Michael envelope is a Fr\'echet space.
\end{abstract}

\maketitle

\markright{Criterion for the local finiteness}

It was discovered in \cite{AHHFG} that topological properties of the Arens-Michael envelope of the semigroup algebra of a countable semigroup~$S$ (denoted by $\wh{\CC S}$) depends on properties of the semigroup and vice versa. Namely, $\wh{\CC S}$ is a Fr\'echet space if and only if  $S$ is finitely generated. In this note we show that $\wh{\CC S}$  is a $(DF)$-space if and only if $S$ is locally finite. We also give a direct proof of an auxiliary result, which asserts that $\wh{\CC S}$ is a nuclear space for every countable~$S$.

Recall that an \emph{Arens-Michael algebra} is a complete Hausdorff locally convex algebra whose topology can be determined by a family of submultiplicative prenorms (i.e., satisfying $\|ab\|\le\|a\|\,\|b\|$ for every elements~$a$ and $b$).
The \emph{Arens-Michael envelope} of a (not necessarily unital) associative algebra~$A$ over $\CC$ is  the completion of~$A$ with respect to the topology determined by all possible submultiplicative prenorms. (This is a simplified version of a definition in~\cite{He93} in which locally convex algebras are considered.) For a countable semigroup  $S$ we denote by $\CC S$  its semigroup algebra over~$\CC$. (Note that  $\CC S$ is unital if and only if~$S$ has an identity.) We need the following  explicit description of the Arens-Michael envelope of $\CC S$ given in \cite{AHHFG}.

We say that an $\R$-valued function $\ell$ on~$S$ is a \emph{length function} if
$$
\ell(s_1s_2)\le \ell(s_1)+\ell(s_2)\qquad (s_1,s_2 \in  S).
$$
Suppose  $\{s_n\!:n\in\N\}$ is a generating subset  (e.g., one can take $S$ itself). For any function $F\!:\N\to \R_+$  put
\begin{equation}\label{locwordlen}
\ell_F(s)\!: = \inf\left\{\sum_{k=1}^m F(n_k)\!: \,
s=s_{n_1}\cdots s_{n_m}\!:\, n_1,\ldots, n_m \in
\N\right\}\,.
\end{equation}
It is easy to see that $\ell_F $ is a non-negative length function.
Note that any element~$a$ of $\CC S$ can be written as a finite sum $\sum_{s\in S}a_s\de_s$, where $\de_s$ is the delta-function at $s$. It is not hard to show that
\begin{equation}\label{prnFdef}
\|a\|_F\!:=\sum_{s\in S}|a_s|\,e^{\ell_F(s)}.
\end{equation}
defines a submultiplicative prenorm on $\CC S$. (Of course, one may replace  here Euler's constant by any strictly positive number.)

\begin{pr}\label{AMCS}\emph{(see \cite[Lemma~3.13]{AHHFG})}
Let~$S$ be a countable semigroup with generating subset $\{s_n\!:n\in\N\}$.
Then the locally convex space
$$
\wh{\CC S}=
\Bigl\{a=\sum_{s\in S} a_s\de_s\! :
\|a\|_F <\infty \quad\forall F\in\R_+^\N\Bigr\}
$$
endowed with the convolution and the topology determined by the family of submultiplicative prenorms $\{\|\cdot\|_F\!:\,F\in\R_+^\N\}$ defined in~\eqref{prnFdef}
is the Arens-Michael envelope of the semigroup algebra~$\CC S$.
\end{pr}

It is worth mentioning  that, by  the universal property of the Arens-Michael envelope \cite[Chapter~5, Definition 2.21]{He93}, the resulting topology does not depend on enumeration of the generating subset.
Note also that we can take $F$ with values in $\N$ but, nevertheless, the corresponding defining system of submultiplicative prenorms may still be uncountable.

To prove the main result, Theorem~\ref{lfDFeq}, we need the following theorem on nuclearity.

\begin{thm}\label{whCSnu}
Let $S$ be a countable semigroup. Then the locally convex space $\wh{\CC S}$ is nuclear.
\end{thm}

An idea of a proof of this theorem was communicated by the author to Akbarov and the corresponding argument (in the group case) can be found in \cite[Lemma 3.15]{Ak20+}.  But in some points it differs from the proof given here.
In particular, Akbarov's version is based on the Grothendieck--Pietsch nuclearity criterion and uses a terminology  different from ours.   The direct and relatively short proof that is presented here includes as well a concise version of a combinatorial trick essential in the argument. Note  that Theorem~\ref{whCSnu} can also be deduced from \cite[Theorem~6.3]{AHHFG}.

For any $F\in\R_+^\N$ we denote by $\CC S_F$ the completion of $\CC S$ with respect to~$\|\cdot\|_F$.
It is obvious that for $F_1\le F_2$ there is a naturally defined continuous homomorphism $\te_{F_2F_1}\!:\CC S_{F_2}\to \CC S_{F_1}$ of Banach algebras.
\begin{pr}\label{F1F1nuc}
For every $F_1\in \Z_+^{\N}$ there is $F_2\in \N^{\N}$ such that $F_2>F_1$ and  $\te_{F_2F_1}\!:\CC S_{F_2}\to \CC S_{F_1}$ is a nuclear linear map.
\end{pr}
Note that it is important in the proof that the generating set is infinite. For a finitely generated semigroup we must take a generating set with repetition.

We need an auxiliary lemma.
\begin{lm}\label{F1F2}
Let $F_1\in \Z_+^{\N}$. Put $F_2(n)\!:=F_1(n)+n$, where $n\in\N$. Then
\begin{equation}\label{SmSmmFF}
\sum_{s\in S}e^{\ell_{F_1}(s)-\ell_{F_2}(s)}<\infty.
\end{equation}
\end{lm}
\begin{proof}
Fix $n\in\N$ and put $C_n\!:=\{s\in S\!:\, \ell_{F_2}(s)-\ell_{F_1}(s)= n\}$. We first claim that
\begin{equation}\label{cardCn}
|C_n|\le 2^n-1,
\end{equation}
where $|C_n|$ denotes the cardinality of $C_n$.

Indeed, let $s\in S$. It follows from the definition of $\ell_{F_2}$ and the fact that the values of $\ell_{F_2}$ are in $\N$ that there are  $n_1,\ldots, n_m \in\N$ such that $s=s_{n_1}\cdots s_{n_m}$ and
$$
\ell_{F_2}(s)=\sum_{k=1}^{m}F_2(n_k).
$$
For every $s\in S$ fix  $n_1,\ldots, n_m$ satisfying this equality. Then we have the map $s\mapsto (n_1,\ldots, n_m)$, where~$m$ depends on~$s$. It is obvious that the map is injective.

Since
$$
\ell_{F_2}(s)=\sum_{k=1}^{m}(F_1(n_k)+n_k)\quad\text{and}\quad
\ell_{F_1}(s)\le \sum_{k=1}^{m} F_1(n_k),
$$
we have $\ell_{F_2}(s)-\ell_{F_1}(s)\ge \sum n_k$ for every~$s$.
In particular, $|C_n|$  does not exceed the number of finite sequences $(n_1,\ldots, n_m)$ in~$\N$ satisfying the inequality $n\ge \sum n_k$. It is well known that the number of compositions of~$j$  (ways of writing~$j$ as the sum $\sum n_k$, where $n_1,\ldots, n_m$ are arbitrary positive integers) equals~$2^{j-1}$. Summing over~$j$, we obtain~\eqref{cardCn}.

Further, $S=\bigcup_{n\in\N} C_n$ because $\ell_{F_2}(s)-\ell_{F_1}(s)>0$ for every~$s$. Since $C_n$ are disjoint, we have
$$
\sum_s e^{\ell_{F_1}(s)-\ell_{F_2}(s)}= \sum_{n=1}^\infty |C_n| e^{-n}.
$$
It follows from \eqref{cardCn} that
$$
\sum_s e^{\ell_{F_1}(s)-\ell_{F_2}(s)}\le \sum_{n=1}^\infty 2^n e^{-n}<\infty.
$$
\end{proof}

\begin{proof}[Proof of Proposition~\ref{F1F1nuc}]
Put $F_2(n)\!:=F_1(n)+n$ as in Lemma~\ref{F1F2}. Let $\chi_s$ denote the functional on $\CC S_{F_2}$ such that $\chi_s(\de_s)=1$ and $\chi_s(\de_t)=0$ when $t\ne s$. It is easy to see that $\|\de_s\|_{F_1}=e^{\ell_{F_1}(s)}$ and $\|\chi_s\|'_{F_2}=e^{-\ell_{F_2}(s)}$, where the prime means that we take the norm in the dual Banach space.
Writing an arbitrary element~$a$ of $\CC S_{F_2}$ as  $\sum_{s} a_s\de_s$ we have
$$
\te_{F_2F_1}(a)= \sum_{s\in S} \chi_s(a)\, \de_s.
$$
It follows from Lemma~\ref{F1F2} that
$\sum \|\chi_s\|'_{F_2} \,\|\de_s\|_{F_1}<\infty$
and so $\te_{F_2F_1}$ is nuclear.
\end{proof}

\begin{proof}[Proof of Theorem~\ref{whCSnu}]
Let  $\{s_n\!:n\in\N\}$ be an infinite generating subset of~$S$.
It follows from Proposition~\ref{AMCS} that $\wh{\CC S}$ is a projective limit of the system of Banach spaces of the form $\CC S_F$, where~$F$ runs over $\R_+^\N$. Moreover, we can assume that each~$F$ takes values in~$\Z_+$.
Then for every submultiplicative prenorm $\|\cdot\|$ on $\CC S$ there is $F_1\in\Z_+^\N$ such that  $\|\cdot\|\le C\|\cdot\|_{F_1}$ for some $C>0$. By Proposition~\ref{F1F1nuc}, for every~$F_1$ there is $F_2\ge F_1$ such that $\te_{F_2F_1}$ is a nuclear linear map, which immediately implies that $\wh{\CC S}$ is nuclear.
\end{proof}

Recall that a semigroup $S$ is said to be \emph{locally finite} if every finite subset of~$S$ is contained in a finite subsemigroup (see, e.g., \cite[p.~161]{Gri95} or, in the case of a group, \cite{KB73}).

Recall also that a locally convex space is called a \emph{$(DF)$-space }if it has a fundamental sequence of bounded subsets and every strongly bounded subset of its strong dual  that is a union of countably many equicontinuous sets is also equicontinuous \cite[\S\,29.3]{Kot1}. We do not use this definition directly but  refer only to the classical results that the strong dual of a Fr\'echet space is a $(DF)$-space and that the strong dual  of a $(DF)$-space is a Fr\'echet space \cite[pp.\,396--397]{Kot1}.

\begin{pr}\label{lofiDF}
Let $S$ be a locally finite countable semigroup. Then $\CC S$ endowed with the strongest locally convex topology is an Arens-Michael algebra. As a corollary, $\wh{\CC S}$ is topologically isomorphic to $\CC S$  and hence it is a $(DF)$-space.
\end{pr}
\begin{proof}
The strongest locally convex topology on $\CC S$ is determined by prenorms of the form $|a|_\phi\!:=\sum_s |\al_s|\,\phi(s)$, where $\phi$ runs over all functions on~$S$ with values in $[1,+\infty)$ (here $a=\sum \al_s\de_s$). On the other hand,  topology on $\wh{\CC S}$ is determined by the family ($\|\cdot\|_F$), where
$F$ runs over all $\R_+$-valued functions on~$\N$ (with respect to a fixed
generating subset $\{s_n\!:n\in\N\}$); see~\eqref{prnFdef}. The former family clearly includes the latter.  To complete the proof we show that the topologies coincide on~$\CC S$. It suffices to show that for any $\phi\!: S\to \R_+$ there is~$F$ such that $\phi(s)\le e^{\ell_F(s)}$ for all $s\in S$.

If $S$ is finite, there is nothing to prove. Suppose that $S$ is  infinite and denote by $S_n$ the subsemigroup generated by $s_1,\ldots,s_n$. Passing to a subsequence, we can assume that $s_n\notin S_{n-1}$. Since $S$ is locally finite, each $S_n$ is finite. Hence there is a non-negative increasing function $F$ such that
\begin{equation}\label{Festl}
F(n)\ge\max\{\ln\phi(s)\!:\,s\in S_n\}\qquad\text{for every~$n$.}
\end{equation}

Fix $s\in S$ and take $n$ such that $s\in S_n\setminus S_{n-1}$. Then  for every decomposition $s=s_{n_1}\cdots s_{n_m}$ there is $n_j$ equal or greater than $n$ and so we obtain $F(n)\le\sum_k F(n_k)$ since $F\ge 0$. Therefore $F(n)\le \ell_F(s)$. It follows from \eqref{Festl} that $\phi(s)\le e^{F(n)}$ and so $\phi(s)\le e^{\ell_F(s)}$. This holds for all $s\in S$ and so the topologies coincide.

Finally, $\CC S$ is isomorphic to the strong dual space of~$\CC^S$, which is a~Fr\'echet space  because~$S$ is countable. Thus $\CC S$ is a $(DF)$-space \cite[p.\,396]{Kot1}.
\end{proof}

\begin{thm}\label{lfDFeq}
A countable semigroup $S$ is locally finite if and only if $\wh{\CC S}$ is a $(DF)$-space.
\end{thm}
\begin{proof}
The necessity follows from Proposition~\ref{lofiDF}.

To verify the sufficiency suppose that $\wh{\CC S}$ is a $(DF)$-space. We need to show that any finitely generated subsemigroup $S_0$ of $S$ is finite. Note that, by the universal property of the Arens-Michael envelope, the embedding $S_0\subset S$ induces a continuous homomorphism $j\!:\wh{\CC S_0}\to\wh{\CC S}$. It follows from the representation of elements of these algebras as series that $j$ is injective. The Hahn–Banach theorem implies that the corresponding linear map $j'\!:(\wh{\CC S})'\to(\wh{\CC S_0})'$ between their strong dual spaces is surjective. Moreover, we claim that $j'$ is open.

To prove the claim we use  Pt\'{a}k's open mapping theorem, which asserts that every surjective continuous  linear map from a Pt\'{a}k space to a barreled space is open (see a proof in \cite[\S\,34.2, p.\,27, (3)]{Kot2} and the definition of a Pt\'{a}k space on p.\,26 in [ibid.]). Since $\wh{\CC S}$ is a $(DF)$-space, the strong dual $(\wh{\CC S})'$ is a Fr\'echet space  \cite[p.\,397]{Kot1}.
It follows from  \cite[\S\,21.10, p.\,273, (5)]{Kot1} that being a Fr\'echet space it is a Pt\'{a}k space. On the other hand, since $S_0$ is finitely generated, $\wh{\CC S_0}$ a Fr\'echet space \cite[Proposition~3.15]{AHHFG}. Being also nuclear by Theorem~\ref{whCSnu}, it is a Montel space \cite[p.\,519, (50.12)]{Tre}.  The strong dual of a Montel space is a Montel space \cite[\S\,27.2, p.\,369]{Kot1}. In particular, $(\wh{\CC S_0})'$ is barreled. Thus we can  apply Pt\'{a}k's open mapping theorem and $j'$ is open.

Finally, since $j'$ is open and $(\wh{\CC S})'$ is a Fr\'echet space,  so is $(\wh{\CC S_0})'$. Note that every metrizable $(DF)$-space is normable \cite[Observation~8.3.6]{CB87}. Thus $\wh{\CC S_0}$ is a Banach space. Then all the norms of the form \eqref{prnFdef} on $\CC S_0$ are equivalent and so  $S_0$ is finite as desired.
\end{proof}

\begin{rem}
Note that the classes of locally finite and finitely generated semigroups are opposite in properties since their intersection contains only finite semigroups. This fact is parallel to the assertion that every
Fr\'echet $(DF)$-space is a Banach space (because every metrizable $(DF)$-space is normable). Indeed, if $S$ is a countable semigroup, $\wh{\CC S}$ is a Fr\'echet space if only if~$S$ is finitely generated \cite[Proposition~3.15]{AHHFG} and $\wh{\CC S}$ is a $(DF)$-space if only if $S$ is locally finite (Theorem~\ref{lfDFeq}). Finally, $\wh{\CC S}$ is a Banach space if only if $S$ is finite.
\end{rem}


\begin{thebibliography}{CCJJVVXX}


\bibitem[Ak20+]{Ak20+}
S.\,S.~Akbarov, \emph{Holomorphic duality for countable discrete groups}, 	arXiv:2009.03372, 2020.

\bibitem[Ar20+]{AHHFG}
O.\,Yu.~Aristov, \emph{Holomorphically finitely generated Hopf algebras and quantum Lie groups}, arXiv:2006.12175, 2020.

\bibitem[CB87]{CB87}
P.~P\'{e}rez Carreras, J.~Bonet, \emph{Barrelled locally convex spaces}, North-Holland, 1987.

\bibitem[Gri95]{Gri95}
P.\,A.~Grillet, \emph{Semigroups: an introduction to the structure theory}, Marcel Dekker, Inc., New York, 1995.

\bibitem[He93]{He93}
A.\,Ya.~Helemskii, \emph{Banach and polynormed algebras: general
theory, representations, homology}, Nauka, Moscow, 1989
(Russian); English transl.: Oxford University Press, 1993.

\bibitem[KB73]{KB73}
O.\,H.~Kegel, B.\,A.\,F.~Wehrfritz,  \emph{Locally finite groups}, North-Holland, Amsterdam, 1973.

\bibitem[Ko69]{Kot1}
G.~K\"{o}the, \emph{Topological vector spaces I},  Springer, New York, 1969.

\bibitem[Ko79]{Kot2}
G.~K\"{o}the, \emph{Topological vector spaces II}, Springer, New York, Heidelberg, Berlin,  1979.


\bibitem[Tr67]{Tre}
F.~Treves, \emph{Topological vector spaces, distributions and kernels}, Academic Press, New York, 1967.
\end{thebibliography}
\end{document}